\documentclass[twoside,12pt,leqno]{amsproc}
\usepackage{amssymb,stmaryrd,latexsym,enumerate,tikz}
\usetikzlibrary{patterns}
\usepackage[pagebackref]{hyperref}
\usepackage{amsrefs}    
\usepackage{etoolbox}   
\hypersetup{citecolor=purple, linkcolor=blue, colorlinks=true}
\usepackage{bookmark}
\numberwithin{table}{section}

\theoremstyle{plain}
\newtheorem{theorem}{Theorem}[section]
\newtheorem{lemma}[theorem]{Lemma}

\theoremstyle{definition} 

\newtheorem{remark}[theorem]{Remark}
\newtheorem{problem}[theorem]{Problem}

\renewcommand{\geq}{\geqslant}
\renewcommand{\leq}{\leqslant}
\renewcommand{\ge}{\geqslant}
\renewcommand{\le}{\leqslant}
\newcommand{\lhdeq}{\trianglelefteqslant}    

\newcommand{\Alt}{\mathrm{Alt}}

\newcommand{\C}{\mathrm{C}}

\newcommand{\Fit}{\mathrm{F}}

\newcommand{\Sym}{\mathrm{Sym}}

\makeatletter        
\def\@adminfootnotes{%
  \let\@makefnmark\relax  \let\@thefnmark\relax
  \ifx\@empty\@date\else \@footnotetext{\@setdate}\fi
  \ifx\@empty\@subjclass\else \@footnotetext{\@setsubjclass}\fi
  \ifx\@empty\@keywords\else \@footnotetext{\@setkeywords}\fi
  \ifx\@empty\thankses\else \@footnotetext{%
    \def\par{\let\par\@par}\@setthanks}%
  \fi}\makeatother   

\begin{document}

\hyphenation{}

\title[Uniformly generated groups]{Classifying uniformly generated groups}
\author[S.\,P. Glasby]{S.\,P. Glasby}

\address[Glasby]{
Centre for Mathematics of Symmetry and Computation,
University of Western Australia,
35 Stirling Highway,
Perth 6009, Australia.\newline
 Email: {\tt Stephen.Glasby@uwa.edu.au; WWW: \href{http://www.maths.uwa.edu.au/~glasby/}{http://www.maths.uwa.edu.au/$\sim$glasby/}}
}

\date{\today\hfill 2010 Mathematics subject classification:
 20E15, 14N20}

\begin{abstract}
  A finite group $G$ is called \emph{uniformly generated}, if whenever there is
  a (strictly ascending) chain of subgroups
  $1<\langle x_1\rangle<\langle x_1,x_2\rangle <\cdots<\langle x_1,x_2,\dots,x_d\rangle=G$, then $d$ is the minimal number of generators
  of~$G$. Our main result classifies
  the uniformly generated groups without using the simple group classification.
  These groups are related to finite projective geometries by a 
  result of Iwasawa on subgroup~lattices.
\end{abstract}

\maketitle

\section{Introduction}\label{S1}
Let $G$ be a finite group. A chain $1=G_0<G_1<\cdots<G_d=G$ of
subgroups of a $G$ is called \emph{unrefinable} if $G_i$ is maximal in
$G_{i+1}$ for each $i$. The \emph{length} of $G$, denoted $\ell(G)$,
is the maximum length of an unrefinable chain, and the \emph{depth} of
$G$, denoted $\lambda(G)$, is the minimum length of an unrefinable
chain. By~\cite{BLS2}, a nonabelian simple group $G$ satisfies
\[
  \lambda(G)\le\left(1+o(1)\right)\frac{\ell(G)}{\log_2(\ell(G))}.
\]
It was shown in~\cite{CST} that
$\ell(\Alt_n)=\left\lfloor\frac{3(n-1)}{2}\right\rfloor-s_2(n)$ where $\Alt_n$
is the alternating group of degree~$n$, and
$s_p(n)=\sum_{i\ge0}n_i$ is the sum of the digits
of the base-$p$ expansion of~$n=\sum_{i\ge0}n_ip^i$.
In~\cite{BLS} and~\cite{BLS2} the length and depth of finite groups, and algebraic groups, are studied. These references review some of the
earlier work in this area.

Iwasawa~\cite{I} proved a striking result, namely
$\ell(G) = \lambda(G)$ if and only if $G$ is supersolvable. Inspired by
this result,~\cite{BLS} classifies the finite groups $G$ for which
$\ell(G) - \lambda(G)$ is `small'. An elementary proof of Iwasawa's result is
given in~\cite{H}*{Theorem~19.3.1}.

We say that $G$ is \emph{$d$-uniformly generated} if for all
$(x_1,x_2,\dots,x_d)\in G^d$ with
\[ 
  1<\langle x_1\rangle<\langle x_1,x_2\rangle  <\cdots<\langle x_1,x_2,\dots,x_d\rangle
\] 
we have $G=\langle x_1,x_2,\dots,x_d\rangle$. In Lemma~\ref{L3}, we
will prove that $G$ is $d$-uniformly generated if and only if $d=\ell(G)$.
In particular, this implies that $G$ can be $d$-uniformly generated for
at most one choice of~$d$. The minimal number of generators of $G$
is denoted~$d(G)$.
Clearly $G=\langle x_1,x_2,\dots,x_d\rangle$ implies $d\ge d(G)$.
Recall that a generating set $S$ for a group $G$ is called \emph{independent}
(sometimes called \emph{irredundant}) if
$\langle S \setminus \{s\}\rangle < G$ for all $s \in S$. Let $m(G)$ denote
the maximal size of an independent generating set for $G$. For example,
$d(\Sym_n)=2$  for $n\ge3$, and $m(\Sym_n)=n-1$ for $n\ge1$ by \cite{W}.
The finite groups
with $m(G)=d(G)$ are classified by Apisa and Klopsch in~\cite{AK}*{Theorem~1.6}.

We say that $G$ is \emph{uniformly generated} if $G$ is $d(G)$-uniformly
generated. By Lemma~\ref{L3}, $G$ is uniformly generated if and only
$d(G) = \ell(G)$. We classify such groups in Theorem~\ref{T1}. Our first
proof of this result  (see~\cite{G}*{p.\,4}) relied on the
Classification of Finite Simple
Groups (CFSG). This dependence seemed undesirable as the conclusion did not involve any
nonabelian simple groups. The proof we give appeals to Iwasawa's result,
and is completely elementary.

\begin{theorem}\label{T1}
Let $G$ be a finite group, and let $\C_n$ denote a cyclic group of order~$n$. Then $G$ is uniformly generated if and only if either $G\cong (\C_p)^d$ is elementary or $G\cong (\C_p)^{d-1}\rtimes\C_q$ where $p,q$ are primes and $\C_q$ acts as a nontrivial scalar on $(\C_p)^{d-1}$.
\end{theorem}

\begin{remark}
  There are two key ideas for the proof of Theorem~\ref{T1}. First,
  for any group $G$, we have $d(G) \leq m(G) \leq \ell(G)$ and
  $d(G) \leq \lambda(G) \leq \ell(G)$, and second
\begin{equation}\label{E1}
  \textup{if $G$ is uniformly generated, then $d(G)=\ell(G)$ and hence
    $\ell(G)=\lambda(G) = m(G)$.}
\end{equation}
Since $\lambda(G)=\ell(G)$, a uniformly generated group $G$ must be
supersolvable by~\cite{I}. Further, since
$d(G) = m(G)$ it is amongst the (solvable)
groups classified by Apisa and Klopsch in~\cite{AK}*{Theorem~1.6}.
Their groups are structurally similar to ours, but with a more general
module action.
Our proof does not refer to~\cite{AK}, even though it would be natural to do so,
because we want our proof to be independent of the CFSG.
\end{remark}

\begin{remark}
The groups we classify in Theorem~\ref{T1} arise in connection with other
very natural characterizations. For example, Iwasawa~\cite{I} classified
the groups $G$ whose subgroup lattice forms a finite projective
geometry with at least three points on a line, and found the same groups.
Further, Baer~\cite{B}*{Theorem~11.2(b)} determined the same groups when
considering ``subgroup-isomorphisms'' and ``ideal-cyclic'' groups~\cite{B}*{p.\,2, p.\,8}.
\end{remark}

\section{Proof}\label{S2}
The characterization of $d$-uniformly generated groups in Lemma~\ref{L3} below
helps to prove Theorem~\ref{T1}.

\begin{lemma}\label{L3}
A finite group $G$ is $d$-uniformly generated if and only if $d=\ell(G)$.
\end{lemma}

\begin{proof}
  The inequality $d\le\ell(G)$ is clear. Suppose now that $G$ is $d$-uniformly
  generated and $d<\ell(G)$. Then there exists an unrefinable chain 
  \[
    1=G_0<G_1<\cdots<G_{\ell(G)}=G.
  \]
  Since $G_i$ is maximal in $G_{i+1}$ we have $G_{i+1}=\langle G_i,x_i\rangle$ for all $x_i\in G_i\setminus G_{i-1}$. It follows that $G_i=\langle x_1,\dots,x_i\rangle$ and $1=G_0<G_1<\cdots<G_d<G_{\ell(G)}=G$. Consequently, $G$ is not $d$-uniformly generated. This contradiction proves the result.
\end{proof}

Recall the following definitions. The \emph{Frattini subgroup}, $\Phi(G)$, is the intersection of the maximal subgroups of $G$; so the elements of $\Phi(G)$ are precisely the elements of $G$ contained in no independent generating sets of $G$. The \emph{Fitting subgroup}, $\Fit(G)$, is the largest normal nilpotent subgroup of $G$.

\begin{lemma}\label{L2}
Let $G$ be a finite uniformly generated group.
\begin{enumerate}[{\rm (a)}]
  \item If $1 \lhdeq N \lhdeq G$, then $N$ and $G/N$ are both uniformly generated.
  \item The Frattini subgroup $\Phi(G)$ is trivial.
\end{enumerate}
\end{lemma}

\begin{proof}
  (a)~Suppose $1\lhdeq N\lhdeq G$. For any group $G$ we have
  $d(G) \leq d(G/N) + d(N)$ and
  $\ell(G) = \ell(G/N) + \ell(N)$, see~\cite{CST}*{Lemma~2.1}. Since $G$ is uniformly generated, 
\[
d(G) = \ell(G) = \ell(G/N) + \ell(N) \geq d(G/N) + d(N) \geq d(G).
\]
Therefore, $\ell(G/N) = d(G/N)$ and $\ell(N) = d(N)$, implying that $G/N$ and $N$ are uniformly generated by Lemma~\ref{L3}.

(b)~Assume that $\Phi(G)\ne 1$, and choose $1\ne x_1\in\Phi(G)$. Suppose $Y=\{y_1,\dots,y_d\}$ generates $G$, where $d=d(G)$. The minimality of $d$ implies $\langle y_2,\dots,y_d\rangle<G$, and hence  $\langle x_1,y_2,\dots,y_d\rangle<G$ as $x_1 \in \Phi(G)$. If for some $i<d$, the subgroup $\langle x_1,y_2,\dots,y_i\rangle$ equals $\langle x_1,y_2,\dots,y_{i+1}\rangle$, then $y_{i+1}\in\langle x_1,y_2,\dots,y_i\rangle$. In this case, we therefore have
\[
  G=\langle y_1,\dots,y_i,y_{i+1},\dots,y_d\rangle
  =\langle y_1,\dots,y_i,x_1,y_{i+2},\dots,y_d\rangle
  =\langle y_1,\dots,y_i,y_{i+2},\dots,y_d\rangle< G.
\]
This contradiction shows that there is a strictly ascending chain
\[
  1<\langle x_1\rangle<\langle x_1,y_2\rangle<\cdots
  <\langle x_1,y_2,\dots,y_d\rangle<G
\]
with too many subgroups, contradicting the fact that $G$ is uniformly generated.
\end{proof}

\begin{proof}[Proof of Theorem~\ref{T1}]
  Assume that $G$ is uniformly generated and $d=d(G)$. 
  Then $\ell(G)=\lambda(G)$ by~\eqref{E1}, and $G$ is supersolvable
  by~\cite{I}.  Assume $G\ne1$ and  $N:=\Fit(G)$.
  Then $N \neq 1$ since $G$ is solvable. 
  Lemma~\ref{L2}(a,b) imply that $\Phi(N)=1$. If $|N|$ is divisible by
  two primes, then we have a smaller generating set.
  Hence~$N$ must be elementary abelian. The first possibility
  is $G=N\cong (\C_p)^d$.  Suppose now that~$N$ is a proper subgroup of~$G$.
  Since $G$ is supersolvable, the derived subgroup $G'$ is nilpotent, so
  $G'\le F(G)$ and $G/F(G)$ is abelian. The above argument shows
  that $G/N$ is an elementary $q$-group. Clearly $q\ne p$.
  Let $g\in G$ have order~$q$.  By Lemma~\ref{L2}(a),
  $N\langle g\rangle$ is uniformly generated, and by Maschke's
  theorem $N$ is a direct sum of simple
  $\langle  g\rangle$-submodules which must
  have dimension~1 and be isomorphic. Therefore $g$ acts as a scalar
  matrix on $N$. The scalar has order~$q$, and not~1 because
  $N=\Fit(G)$. Also, if $|G/N|=q^k$, then we must have $k=1$,
  otherwise we could find an element of order $q$ centralizing
  $N$ and hence $N<\Fit(G)$, a contradiction. In summary, either
  $G\cong(\C_p)^d$ or $G\cong(\C_p)^{d-1}\rtimes\C_q$ where $\C_q$ acts as a
  nontrivial scalar on $(\C_p)^{d-1}$. Conversely, such groups are easily shown
  to be uniformly generated and to~have~$d=d(G)$.
\end{proof}

We conclude with two open problems.

\begin{problem}
  Classify the finite groups $G$ with $m(G)-d(G)\le1$.
\end{problem}

\begin{problem}
  Bound the difference $m(G)-d(G)$, for a connected algebraic group~$G$.
\end{problem}

\section*{Acknowledgments}

The problem of classifying uniformly generated groups was posed by the author
at the 2018 CMSC Annual Research Retreat, and solved promptly. I thank the
CMSC for hosting the Retreat, and Scott Harper for his helpful comments. I
acknowledge the support of the Australian Research Council Discovery Grants
DP160102323 and DP190100450. Finally, I thank the referee for suggesting
improvements to this note.

\end{document}